\documentclass[a4paper,12pt]{article}
\usepackage{times, url}
\usepackage{amsmath,amssymb,amsthm,amsfonts}

\newtheorem{theorem}{Theorem}[section]

\newtheorem{algorithm}[theorem]{Algorithm}

\newtheorem{remark}[theorem]{Remark}

\begin{document}
\setcounter{page}{1} 
%\noindent {\bf Notes on Number Theory and Discrete Mathematics \\ 
%Print ISSN 1310--5132, Online ISSN 2367--8275 \\
%Vol. XX, XXXX, No. X, XX--XX}
%\vspace{10mm}

\begin{center}
{\LARGE \bf  Propagating Graceful Graphs and Trees}
\vspace{8mm}

{\Large \bf Keneth Adrian Dagal$^1$ and Kristoffer Karan Hugo$^2$}
\vspace{3mm}

$^1$ Nasser Vocational Training Centre \\ 
Jau, Kingdom of Bahrain \\
e-mail: \url{kendee2012@gmail.com}
\vspace{2mm}

$^2$ National University Laguna\\ 
 Laguna, Philippines\\
e-mail: \url{kristofferkghugo@gmail.com}
\vspace{2mm}

\end{center}
\vspace{10mm}

\noindent
%{\bf Received:} DD Month YYYY \hfill {{\bf Revised:} DD Month YYYY	\hfill {Accepted:} DD Month YYYY \\ \\ \noindent
{\bf Abstract:} In an attempt to prove the Graceful Tree Conjecture, we present two propagation of graphs. The first is to propagate graceful graphs, and the second is to propagate trees from a gracefully labeled tree. The motivation in propagating such graphs is to see how graphs behave in the lens of their adjacency matrices.  Thus, we provide the necessary algorithms for the said objectives. \\

\vspace{5mm}

\section{Introduction} 
A graph $G$ with $m$ edges is said to be \textit{graceful} if the vertices can be labeled using the positive integers 1,2, …, $m+1$, without repetition, such that for each given edge, a weight that is equal to the absolute difference of the numbers assigned to the vertices joining the edge, the edge weights runs from 1 to $m$.\\

The adjacency matrix $A_G$ of a graph G with $n$ vertices, whose vertices $v_i \in V$, for $i$ from $1$ to $n$, is the
 $n \times n$ matrix $[a_{ij}]$ such that
\begin{center}
    $a_{ij} = \begin{cases} 
      1& v(i,j) \in E \\
      0 & v(i,j) \notin E. \\
   \end{cases}$
\end{center}

Here are some results of Bloom \cite{Bloom} about gracefully labeled adjacency matrix.

\begin{theorem}\label{Theorem1}
If $G_g$ is the set all of gracefully labeled simple graphs G with $n$ vertices and $m=n-1$ edges, then there are $\mid G_g \mid =m! = (n-1)!$ gracefully labeled $G$.
\end{theorem}

\begin{theorem}\label{Theorem2}
The adjacency matrix $A_g = [a_{ij}]$ is said to be a \textit{gracefully labeled adjacency matrix} if,  for each $c$th diagonal, there exist exactly one "1" for each $c= i-j \neq 0$ and all zeroes for $c=0$.
\end{theorem}

The problem of determining which graphs are graceful evolved from the 1963 conjecture of Ringel and Kotzig regarding the decomposition of the complete graph on $2n+1$ vertices into isomorphic copies of any tree of length $n$. Bloom \cite{Bloom} has shown that the problem can be solved by showing that all trees are graceful.

While the problem remains unsolved, a direction of research is to discover families of graceful trees and methods of constructing graceful trees. Most families of graceful trees were discovered by trial and error. We can expect that graphs, in general, become more challenging to handle as the number of vertices of the graph increases. Adjacency matrices describe many of the corresponding graphs' characteristics and properties. Interest in adjacency matrices is motivated by these properties when the corresponding graph is labeled gracefully. Several families of trees are shown to be graceful by building the adjacency matrix to be graceful.

Some progress in broadening the class of trees known to be graceful has been achieved by examining the adjacency matrices of graceful trees. Preliminary independent observations of the characteristics of the adjacency matrices of graceful trees were established by Haggard and McWha\cite{Haggard} and Bloom \cite{Bloom2} . In Bloom's paper, the adjacency matrices of graceful trees may be used to generate new graceful trees. The Stanton and Zarnke\cite{Stanton}  techniques can be shown to be special cases of assembling the adjacency matrices of graceful trees into larger adjacency matrices of larger graceful trees.

Cavalier \cite{Cavalier}  characterized the adjacency matrix of a graceful graph and families of graceful trees. Loyola \cite{Loyola}  discovered three new methods of constructing graceful trees and three families of graceful trees.
\section{Propagating Graceful Graphs}

We now discuss the two algorithms generating graceful graphs by adding a vertex or vertices.

\begin{algorithm}{Single-vertex propagation of graceful graphs.}
\begin{enumerate}
    \item Consider a graceful graph $G$ with $n$ vertices and $m = n-1$ edges and use the gracefully labeled  adjacency matrix of $G$ in the form $$A_g = [a_{ij}].$$ 
    \item Extend the main diagonal $D_c=D_0$ with all zero entries of $A_g$ by augmenting the entry $a_{(n+1)(n+1)} = 0$ as shown below:
\begin{center}
$
 \begin{bmatrix}
A_g & \\
 & 0 
\end{bmatrix}  
$
\end{center}

\item Insert two diagonals in $A_g$. The new $1$st diagonal is $D_{1}’ = \{b_{12},b_{23}, . . .,b_{n(n+1)}\}$ above  $D_0$ and the new $-1$st diagonal $D_{-1}’ = \{b_{21},b_{32},…,b_{(n+1)n}\}$ below it, to form the $(n+1)\times (n+1)$ symmetric transition matrix $T(A_g)$ as shown below:

$$
T(A_g)=\left[
\begin{array}{ccccccc}
0 & b_{12} & a_{12} & a_{13} & \cdots & \cdots & a_{1n}\\
b_{12} & 0 & b_{23} & a_{23} & \cdots & \cdots & a_{2n}\\
 \vdots & \ddots & \ddots & \ddots & \ddots & \ddots & \vdots\\
 \vdots & \ddots & \ddots & \ddots & \ddots & \ddots & \vdots\\
 \vdots & \ddots & \ddots & \ddots & \ddots & \ddots & \vdots\\
 a_{(n-1)1} & a_{(n-1)2} & \cdots &\ddots & \ddots &0 &b_{(n)(n+1)} \\
a_{n1} & a_{n2} & a_{n3} &\cdots & \cdots & b_{(n+1)(n)} & 0\\
\end{array}
\right]
$$
\item For $i = 1$, set $b_{i(i+1)}$ to 1 and all other entries of $D_1’ = D_{-1}’$ to zero to obtain the updated $T_i(A_g)$.
\item Use $T_i(A_g)$ to generate the graph $G_i’$ on $n+1$ vertices and $m+1 = n$ edges.
\item Repeat Steps 4 and 5 for $i = 2,3,\cdots ,n$.
\end{enumerate}
\end{algorithm}

\begin{theorem}
All graphs generated by Algorithm 2.1 are graceful.
\end{theorem}

\begin{proof}
Let $G$ be a gracefully labeled simple graph on $n$ vertices and $n-1$ edges and $A_g$ be the gracefully labeled adjacency matrix of $G$.  By theorem 1.2, only one 1 for each diagonal of $A_g$ exists except for the main diagonal $D_0$ where all entries are 0.

Let $T_i(A_g)$ be a graph generated by Algorithm 2.1 when applied to $G$.  By construction, $T_i(A_g)$  is a simple graph on $n+1$ vertices and $n$ edges. Also, $b_{i(i+1)} = 1 $ and $ b_{k(k+1)} = 0$ for each $k \neq i$.  Thus, there exists one 1 for diagonals $D_{1}’$  and $D_{-1}’$ of $T_i(A_g)$.  Since the other diagonals of $T_i(A_g)$ are precisely the diagonals of $A_g$, then $T_i(A_g)$  is graceful. 
\end{proof}

Note that applying Algorithm 2.1 for each of the $m! = (n-1)!$ graceful graphs $G$ with $n$ vertices and $m = n-1$ edges generate $n$ graceful graphs $G_i$’ on $n+1$ vertices and $m+1 = n$ edges. Thus, Algorithm 2.1 generates all the $n(n-1)! = (m+1)! = n!$ graceful graphs on $n+1$ vertices and $ m+1 = n $ edges. Recursive applications of Algorithm 2.1 exhaustively yield graceful graphs with the desired number of vertices or edges.

\begin{algorithm}{Multiple-vertex propagation of graceful graphs.}
\begin{enumerate}
    \item Consider a graceful graph $G$ with $n$ vertices and $m = n-1$ edges and use the gracefully labeled  adjacency matrix of $G$ in the form $$A_g = [a_{ij}].$$ 
    \item Augment $n_e$ variable rows immediately below  and $n_e$ variable columns immediately to the right of $A_g$ to form the $(n+n_e)\times(n+n_e)$ transition matrix $T(A_g)$. The main diagonal elements of $A_g$ are all zero. Set block $D$ of $T(A_g)$ to zero so that $T(A_g)$ takes the following form:
$$
 T(A_g) =\left[
 \begin{array}{cc}
A_g & B \\
B & D 
\end{array} 
\right]=\left[
 \begin{array}{cc}
A_g & B \\
B & 0
\end{array} 
\right]
$$

\item Since every diagonal except at $c=0$ of $A_g$, transform $B$ into two forms: $B_L$ as a lower triangular matrix and $B_U$ as an upper triangular matrix such that each nonzero diagonal of the matrix has exactly one 1 and all others are zero.
$$
T(A_g) =\left[
 \begin{array}{cc}
A_g & B_U \\
B_L & 0
\end{array} 
\right]
$$
\end{enumerate}
\end{algorithm}

\begin{theorem}
All graphs generated by Algorithm 2.3 are graceful.
\end{theorem}

\begin{proof}
Let $G$ be a gracefully labeled simple graph on $n$ vertices and $n-1$ edges and $A_g$ be the gracefully labeled adjacency matrix of $G$.  By theorem 1.2, only one 1 for each diagonal of $A_g$ exists except for the main diagonal $D_0$ where all entries are 0.

Let $G_i’$ be a graph generated by Algorithm 2.3 when applied to $G$. By construction, $G_i’$ is a simple graph on $n+n_e$ vertices and $n+n_e-1$ edges.

If a diagonal of $T(A_g)$ contains an element of $A_g$, then there exists one 1 since those 1's are in $A_g$ and the remaining elements of the diagonal are set to 0. If a diagonal of  $T(A_g)$ does not contain an element of AG, then there exists one  1 since exactly one of its elements is set to 1 and the remaining elements of the diagonal are set to 0.

By Theorem 1.2, the resulting graph is graceful.

\end{proof}
\begin{remark}
Algorithm 2.3 shows a way to propagate a graceful graph from a certain graceful graph such that this graph is a subgraph of the propagated graceful graph. This algorithm is not exhaustive unlike Algorithm 2.1.
\end{remark}

\section{Propagating trees from a gracefully labeled tree}

For this section, we fix a gracefully labeled $n$-vertex tree $T$ with an $n \times n$ gracefully labeled adjacency matrix $A_t$.  Our objective is to generate all trees with $n+1$-vertex such that $T$ is a subgraph.

\begin{algorithm}{Propagating all trees with 1 more vertex from a given graceful tree}
\begin{enumerate}
    \item Consider a graceful tree $T$ with $n$ vertices and $m = n-1$ edges and use the gracefully labeled  adjacency matrix of $T$ in the form $$A_t = [a_{ij}].$$
    
    \item Add the vertex $v_{n+1}$ by transforming $A_t$ to the transition matrix :
$$
T(A_t) =\left[
 \begin{array}{cc}
A_t & V \\
V^T & 0 
\end{array} 
\right]
$$

where $V=[v_{ij}]$ is an $n \times 1$ matrix with $v_{i1}=1$ for $i=1$ and all else  are $0$'s.
\item Relabel $v_{n+1}$ to $v_{n+2-i}$ and relabel all vertex-labeled $v_s$ more than $v_{n+2-i}$ to $v_{s+1}$.

\item Repeat Steps 2 and 3 for $i = 2,3,\cdots ,n$.

\end{enumerate}
\end{algorithm}

\begin{theorem}{ Generating all trees.}

Let $T_n$ be a tree with $n$ vertices. Algorithm 3.1 generates all trees.
\end{theorem}

\begin{proof}
We prove by induction. For $T_n$, $n=1,2,3$, there is only one possible tree. The tree is graceful by verification. We start the inductive step from $n=3$ to $n=4$. By adding one vertex and employing Algorithm 3.1, we generate all graceful trees of $T_4$. 

In general, using Algorithm 3.1 with all trees of $T_k$ generates all trees of $T_{k+1}$. This is exhaustive in such a way that we connect an edge to every labeled vertex of a given tree $T_k$. Note that for every additional vertex, there must only be one additional edge. Also, the algorithm eliminates the cyclic possibility. Thus, this completes the proof.
\end{proof}

\section{Concluding Remarks}

Algorithm 2.1 generates all graceful graphs but not necessarily a tree, while Algorithm 3.1 generates all trees but not necessarily graceful. For Algorithm 2.1, finding which adjacency matrix is not a tree requires checking whether there are no cycles or no isolated vertex. However, it is not sufficient to deduce that the graph is a tree by simply having no isolated vertex because a disjoint graph is a possibility. For Algorithm 3.1, since we know that all adjacency matrices constitute trees, we only need either a combinatorial argument to ensure graceful property and prove the graceful tree conjecture or a contradiction proof. 

Currently, the authors are looking into this direction to prove the conjecture.

%\section*{Acknowledgements} 

\makeatletter
\renewcommand{\@biblabel}[1]{[#1]\hfill}
\makeatother

\end{document}